\documentclass[11pt]{amsart}
\usepackage{geometry}                % See geometry.pdf to learn the layout options. There are lots.
\usepackage{graphicx}
\usepackage{amssymb}
\usepackage{epstopdf}
\usepackage{color}
\newtheorem{theorem}{Theorem}[section]
\newtheorem{lemma}[theorem]{Lemma}

\theoremstyle{definition}

\newtheorem{proposition}[theorem]{Proposition}

\numberwithin{equation}{section}

\def\R{\mathbb{R}}

\def\aa{{\alpha}}
\def\bb{{\beta}}

\def\th{{\theta}}

\def\dm{{\diamond}}

\newcommand{\Z}{\mathbb{Z}}

\newcommand{\p}{\partial}
\newcommand{\wt}{\widetilde}

\newcommand{\BFR}{\mathbf{R}}
\newcommand{\BFN}{\mathbf{N}}

\newcommand{\ep}{\varepsilon}
\newcommand{\BBN}{\mathbb{N}}

\newcommand{\be}{\begin{equation} }
	\newcommand{\ee}{\end{equation}}
\newcommand{\bse}{\begin{subequations}}
	\newcommand{\ese}{\end{subequations}}
\def\bea{\begin{eqnarray}}
	\def\eea{\end{eqnarray}}

\title[Non-contractible non-self-intersected closed geodesics]{Non-contractible closed geodesics on compact Finsler space forms without self-intersections}

\author{Yuchen Wang${}^\dagger$}
\address{School of Mathematics, Tianjin Normal University,\\ 300384, Tianjin, China}
\email{wangyuchen@mail.nankai.edu.cn}
\thanks{${}^\dagger$ Partially supported by the National Science Foundation of China No. 11831009, 12131012 and the funding of innovating activities on Science and Technology of Hubei Province. }

\date{}

\begin{document}
	
\begin{abstract}
Let $M=S^n/ \Gamma$ and $h \in \pi_1(M)$ be a non-trivial element of finite order $p$, where the integers $n, p\geq2$ and $\Gamma$ is a finite abelian group which acts on the sphere freely and isometrically, therefore $M$ is diffeomorphic to a compact space form which is typical a non-simply connected manifold. We prove there exist at least two non-contractible closed geodesics on $\mathbb{R}P^2$ and obtain the upper bounds on their lengths. Moreover, we prove there exist at least $n$ prime non-contractible simple closed geodesics on $(M,F)$  of prescribed class $[h]$, provided  
\[
F^2 <(\frac{\lambda+1}{\lambda})^2 g_0 \;\; \text{ and }	 \;\; (\frac{\lambda}{\lambda+1})^2 < K \leq 1 \text{ for $n$ is odd or }\; 0<K \leq 1 \text{ for $n$ is even},
\]	
where $\lambda$ is the reversibility, $K$ is the flag curvature and $g_0$ is standard Riemannian metric. Stability of these non-contractible closed geodesics is also studied.
\vspace{0.5 cm}

		

\end{abstract}
	
\maketitle

\section{Introduction} \label{S:Intro}

Let $\Gamma$ be a finite abelian group acting on the spheres $S^n$ freely and isometrically and $h \in \Gamma$ be a non-trivial element of order $p \geq 2$. The quotient space $M:=S^n/\Gamma$ is a smooth manifold which is referred to compact space forms. It is clearly a non-simply connected manifold with $\pi_1(S^{n}/\Gamma)=\Gamma$. We interest in {\it closed geodesics} on the Finsler manifold $(M,F)$, where $F$ denotes the Finsler metric on $M$ with the {\it reversibility} 
\[
\lambda:=\max\{ F(-X) \mid X \in TM, \; F(X)=1\} \geq 1. 
\]
In particular, $F$ is reversible when the equality holds. Closed geodesics are closed curves on the maifold which are the shortest path connecting any two nearby points, where the length of a curve $\gamma:[a,b] \to M$ is given by 
\be \label{E:length}
L(\gamma) = \int_a^b F(\gamma,\dot{\gamma}) dt.
\ee
It is well known (cf. Chapter 1 of Klingenberg \cite{Kli78}) that $c$ is a closed geodesic or a constant curve on $(M,F)$ if and only if $c$ is a critical point of the energy functional 
\begin{equation}\label{energy}
	E(\gamma)=\frac{1}{2}\int_{0}^{1}F(\gamma,\dot{\gamma})^{2}dt
\end{equation}
on the free loop space 
\begin{equation*}
	\label{LambdaM}
	\Lambda M=\left\{\gamma: S^{1}\to M\mid \gamma\ {\rm is\ absolutely\ continuous\ and}\
	\int_{0}^{1}F(\gamma,\dot{\gamma})^{2}dt<+\infty\right\}
\end{equation*}
endowed with a natural structure of Riemann-Hilbert manifold, where the group $S^1=\R/\Z$ acts continuously by isometries (cf. Shen \cite{Shen01}). This formulation indicates that 
Morse-theoretic approaches could be implemented. But one has to overcome several essential difficulties, for example, to distinguish prime closed geodesics and their iterations.

% Moreover, 

A closed geodesic $c:S^1=\mathbb{R}/\mathbb{Z}\to M$ is {\it prime} if it is not a multiple covering (i.e., iteration) of any other
closed geodesics. Here the $m$-th iteration $c^m$ of $c$ is defined by $c^m(t)=c(mt)$ and the inverse curve $c^{-1}$ of $c$ is defined by
$c^{-1}(t)=c(1-t)$ for $t \in \mathbb{R}$. $c$ is not self-intersected or {\it simple} if its image is diffeomorphic to a circle. It worths to point out that the inverse curve of a closed geodesic $c$ need not be a geodesic unless $\lambda =1$. Two prime closed geodesics
$c$ and $d$ are {\it distinct} if there is no $\theta \in (0,1)$ such that
$c(t)=d(t+\theta )$ for all $t\in \mathbb{R}$. We shall omit the word {\it distinct} when we talk about more than one prime closed geodesic. We denote the linearized Poincar\'{e} map of closed geodesic $c$ by the linear symplectic diffeomorphism  $P_c \in Sp(2n-2)$. A Finsler metric $F$ is {\it bumpy} if all the closed geodesics on $(M, F)$ are non-degenerate, i.e., $1\notin \sigma(P_c)$ for any closed geodesic $c$.  For any $M \in Sp(2k)$, the elliptic height $e(M)$ of $M$ is the total algebraic multiplicity of all eigenvalues of $M$ on the unit circle $\mathbf{U}= \{z \in \mathbb{C} ||z|=1 \}$. $e(M)$ is even and $0 \leq e(M) \leq 2k$. A closed geodesic $c$ is non-hyperbolc if $e(P_c) \geq 2$. \\

Existence and multiplicity of closed geodesics is a long-standing issue in the global differential geometry and calculation of variations. The first mathematical rigorous result is due to Birkhoff \cite{Bir27} where he proved the existence of non-trivial closed geodesics on closed surfaces with arbitrary Riemannian metrics. Lyusternik and Fet extended this result for higher dimensional Riemannian manifolds in \cite{LF51}. Their argument also carries on the Finsler metrics. The multiplicity of closed geodesics, on the other hand, is much more subtle and the topological structure of free loop spaces is essentially involved. In a seminal work \cite{GM69}, Gromoll and Meyer obtained infinitely many geometrically distinct closed geodesics on Riemannian manifolds, provided the sequence of Betti numbers $\{b_p(\Lambda M; \mathbb{Q})\}_{p\in\mathbb{N}}$ of the free loop space $\Lambda M$ is unbounded.  Vigu$\acute{e}$-Poirrier and
Sullivan \cite{ViS} pointed out that the Betti number sequence of simply-connected closed manifold is bounded if and only if $M$ satisfies
\be
H^*(M; \mathbb{Q})\cong T_{d, n+1}(x)=\mathbb{Q}[x]/(x^{n+1}=0)\nonumber
\ee
with a generator $x$ of degree $d\geq 2$ and height $n+1\geq 2$, where $\dim M = dn$.
Thereafter, main interests focus on the manifolds called compact rank one globally symmetric spaces where the Gromoll-Meyer assumption does not hold, including
\begin{gather}
	S^n,\quad \mathbb{R}P^n,\quad \mathbb{C}P^n,\quad \mathbb{H}P^n \quad and \quad CaP^2.\nonumber
\end{gather}
A long-standing conjecture relates there are infinitely many distinct closed geodesics on every compact simply-connected Riemannian manifold. The answer is positive for $S^2$ but others is widely open. We refer to \cite{Ban93}, \cite{Fra92} and \cite{Hin93} for more details.  \\

Very surprisingly, the answer is negative for Finsler manifolds due to Katok's example \cite{Kat73} which is an irreversible and bumpy Finsler metrics but only carries $2[\frac{n+1}{2}]$ distinct closed geodesics. See also in \cite{Zil83}. There is a solid belief that Katok's example indeed gives the optimal lower bounds of distinct closed geodesics. In 2004, Bangert and Long \cite{BL10} (published in 2010) proved there are at least two distinct closed geodesics on every Finsler 2-sphere. Their argument is based on the Maslov-type index theory for symplectic paths. Since then a great number of results on the multiplicity of closed geodesics on simply connected Finsler manifolds appeared, for which we refer readers to \cite{DuL1}-\cite{DLW2}, \cite{HiR}, \cite{LoD}, \cite{Rad04}-\cite{Rad10},\cite{Tai1}, \cite{Wan1}-\cite{Wan2} and the references therein. 

Besides fruitful results on simply connected manifolds where closed geodesics are {\it contractible}, we are aware that there are not many papers on the multiplicity of closed geodesics on non-simply connected manifolds whose free loop space possesses bounded Betti number sequence, at least when they are endowed with Finsler metrics. For example, Ballman et al. \cite{BTZ81} proved in 1981 that every Riemannian manifold with the fundamental group being a nontrivial finitely cyclic group and possessing a generic metric has infinitely many distinct closed geodesics. In 1984, Bangert and Hingston \cite{BH84} proved that any Riemannian manifold with the fundamental group being an infinite cyclic group has infinitely many distinct closed geodesics.

In order to apply Morse theory to the multiplicity of closed geodesics, motivated by the studies on the simply connected manifolds, in particular, the resonance identity proved by Rademacher \cite{Rad89}, Xiao and Long \cite{XL15} in 2015 investigated the topological structure of the non-contractible loop space and established the resonance identity for the non-contractible closed geodesics on $\mathbb{R}P^{2n+1}$ by using of
$\Z_2$ coefficient homology. As an application, Duan, Long and Xiao \cite{DLX15}
proved the existence of at least two distinct non-contractible closed geodesics on $\R P^{3}$ endowed with a bumpy and irreversible Finsler metric. Subsequently in \cite{Tai16}, Taimanov used a quite different method from \cite{XL15} to compute the rational equivariant cohomology of the non-contractible loop spaces of compact space forms $S^n/ \Gamma$ and
proved the existence of at least two distinct non-contractible closed geodesics on $\mathbb{R}P^2$ endowed with a bumpy irreversible Finsler metric, and there are at least two non-contractible closed geodesics of class $[h]$ on the compact space form $M=S^n/ \Gamma$ if $\Gamma$ is an abelian group, $h$ has an even order and is nontrivial in $\pi_1(M)$
and $\pi_1(\Lambda_{h} (M))_{SO(2)}\neq 1$. In \cite{Liu17}, Liu combined the Fadell-Rabinowitz index theory with Taimanov's topological results to obtain several multiplicity results of non-contractible closed geodesics on positively curved Finsler $\mathbb{R}P^n$. The result has been improved to compact Finsler space form $(S^n/\Gamma,F)$ soon in \cite{LW22} by Liu and the author. In \cite{LX},  Liu and Xiao established
the resonance identity for the non-contractible closed geodesics on $\mathbb{R}P^n$, and
proved the existence of at least two distinct
non-contractible closed geodesics on every bumpy $\mathbb{R}P^n$ with $n\geq2$ together with \cite{DLX15} and \cite{Tai16}.
Furthermore, Liu, Long and Xiao \cite{LLX18} established the resonance identity for non-contractible closed geodesics of
class $[h]$ on compact space form $M=S^n/ \Gamma$ and obtained at least two non-contractible closed geodesics of class $[h]$ provided $\Gamma$ is abelian and $h$ is nontrivial in $\pi_1(M)$.  Based on the resonance identity, Liu and the author proved that there are infintely many non-contractible closed geodesics of class $[h]$ for $C^\infty$-generic Finsler metrics in a recent paper \cite{LW2023}.  

Closed geodesics are indeed Reeb flows on the unit sphere bundle therefore multiplicity results of Reeb closed orbits holds for closed geodesics straighforward. On the other hand, closed geodesics has a very strong geometry reality, for example, closed geodesics can be knotted in serveral ways. Therefore the existence and multiplicity of closed geodesics with prescribed self-intersections, especially for those which are unknotted, is a very interesting problem. Lusternik and Schnirelmann claimed there exists three simple closed geodesics for any Riemannian metric on two-sphere in 1929. The proof was completed by Ballmann. A higher-dimensional analogy was claimed by Alber but the complete proof was given by Ballmann, Thorbergson and Ziller \cite{BTZ82}, Anosov \cite{Ano1980} and Hingston \cite{Hin84}. In \cite{Wan12}, Wang obtained $n$ distinct simple closed geodesics on  Finsler spheres with suitable geometry conditions. It is well-known that a simple closed geodesics can always be found by taking the shortest homotopically non-trivial closed geodesics, provided the fundamental group is non-trivial. On the other hand, up to the author's knowledge, there seems few result on the multiplicity non-contractible simple closed geodesics on Finsler manifolds.

Note that the only non-trivial finite group acting on $S^n$ freely and isometrically is $\mathbb{Z}_2$ provided $n$ is even, thus $M$ is indeed the real projective spaces for even $n$. Our first result is on the multiplicity of non-contractible closed geodesics on real projective spaces.
\begin{theorem} \label{T:main-1}
There exist at least two distinct non-contractible closed geodesics on $(\mathbb{R}P^2, F)$ with reversibility $\lambda$ and the flag curvature $K$ satisfying $\left(\frac{\lambda}{\lambda+1} \right)^2<\delta \leq K \leq 1$, where $c_1$ is the minimal energy closed curve on $\Lambda_h M$ with 
\[
L(c_1) < \frac{\pi}{\sqrt{\delta}} \quad \text{ and } \quad  
	L(c_2) \leq  \frac{\pi}{\sqrt{\delta}} \left(2 + \frac{1}{\sqrt{\delta}\frac{\lambda+1}{\lambda}-1}\right).
\]
	%Moreover,  %and the number of distinct non-contractible closed geodesics whose length is $L$ at most is bounded from below by $\frac{\frac{L}{2\pi}\sqrt{\delta} -1}{ \frac{\lambda}{\lambda+1} \frac{L}{\pi} + \frac{1}{2}}  (n-1)$.
\end{theorem}

%\begin{remark}
%This theorem is motiviated by the work \cite{Rad07}, where we carefully estimate the quantities with the assistance of precise iteration formula \cite{lo2002}.
%\end{remark}
\noindent When the Finsler metrics are not far from standard metric (Riemannian metrics with constant curvature 1), we have the follows results% on the  multiplicity of non-contractible geodesics on the compact Finsler space forms without self-intersections 
\begin{theorem} \label{T:main-2}
On every compact Finsler space form $(M,F)$ with reversibility $\lambda$ holding with   
\[
F^2<(\frac{\lambda+1}{\lambda})^2 g_0, \quad 
(\frac{\lambda}{\lambda+1})^2 < K \leq 1 \text{ for $n$ is odd or } 0<K \leq 1 \text{ for $n$ is even},  
\]
there exists at least $n$ minimal non-contractible simple closed geodesics of class $[h]$.
\end{theorem}
\noindent Linear stability of non-contractible closed geodesics are also studied.
\begin{theorem} \label{T:main-3}
On every Finsler compact space form $(M,F)$ with reversibility $\lambda$ satisfying $F^2<\frac{1}{\rho} g_0$ and $\left(\frac{\lambda}{\lambda+1}\right)^2  \leq \delta < K \leq 1$ for $n$ is odd, there exists at least 
\[
n - \left[ \frac{n-1}{p} \frac{1}{\sqrt{\rho}}+1 \right] + \left[\frac{n-1}{p} \frac{\lambda+1}{2\lambda} \sqrt{\delta}+1\right]
\]
non-hyperbolic non-contractcible closed geodesics of class $[h]$. 
\end{theorem}

%\begin{remark}
%In particular, under the assumption in Theorem \ref{T:main-2}, there exists at least 
%\[
%n- [\frac{1}{2} \frac{n-1}{p}]
%\] 	
%non-hyperbolic non-contractible closed geodesics without self-intersections on %$(S^n/\Gamma,F)$ of class $[h]$.
%\end{remark}

This paper is organized as follow. In Section \ref{S:Critical}, we give necessary backgrounds on the closed geodesics. Section \ref{S:2-dim} is devoted to the proof of Theorem \ref{T:main-1}. In Section \ref{S:Equivariant}, we apply the equivariant Morse theory to study the non-contractible closed geodesics of given class $[h]$ and finish the proof of Theorem \ref{T:main-2} and Theorem \ref{T:main-3}.

\section{Preliminaries} \label{S:Critical} 

In this section we gives necessary backgrounds concerned with non-contractible closed geodesics, including Morse theory for non-contractible closed geodesics and the precise iteration formula for closed geodesics. 

\subsection{Morse theory of non-contractible closed geodesics} 
For a non-simply connected manifold, the free loop space is decomposed as disjoint components
\[
\Lambda M = \bigsqcup_{ h \in \pi_1(M)} \Lambda_h M.
\]
The cohomology of the non-contractible connected component of the free loop space $\Lambda_h M$ is due to Taimanov.
\begin{lemma}[Theorem 3 of \cite{Tai16} or Lemma 2.3 of \cite{LLX18}] \label{L:Betti-Tai}
	Let $\Gamma$ be a finite abelian group. For $M=S^n/\Gamma$, and a nontrivial $h \in \pi_1(M)$ we have
	\begin{enumerate}
		\item When $n=2k+1$ is odd, the $S^1$-cohomology ring of $\Lambda_h M$ has the form
		$$H^{S^1, *}(\Lambda_h M; \mathbb{Q})=\mathbb{Q}[w, z]/ \{w^{k+1} = 0\}, \quad deg(w)=2, \quad deg(z)=2k$$
		Then the $S^1$-equivariant Poincar$\acute{e}$ series
		of $\Lambda_h M$ is given by
		\begin{equation}
			\begin{split}
				P^{S^1}(\Lambda_h M; \mathbb{Q})(t) & =\frac{1-t^{2k+2}}{(1-t^2)(1-t^{2k})} \\
				&=\frac{1}{1-t^2}+\frac{t^{2k}}{1-t^{2k}}\\
				&=(1+t^2+t^4+\cdots+t^{2k}+\cdots)+(t^{2k}+t^{4k}+t^{6k}+\cdots),
			\end{split}\nonumber
		\end{equation}
		which yields Betti numbers
		\be
		b_q = \mathrm{rank} H_q^{S^1}(\Lambda_h M;\mathbb{Q}) = \begin{cases}
			2,&\quad {\it if}\quad q\in \{j(n-1)\mid j\in\mathbb{N} \} \\
			1,&\quad {\it if}\quad q\in2\mathbb{N}_0\setminus \{j(n-1)\mid j\in\mathbb{N}\} \\
			0, &\quad {\it otherwise}.
		\end{cases}
		\ee
		
		\item When $n=2k$ is even, the $S^1$-cohomology ring of $\Lambda_h M$ has the form
		$$H^{S^1, *}(\Lambda_h M; \mathbb{Q})=\mathbb{Q}[w, z]/ \{w^{2k} = 0\}, \quad deg(w)=2, \quad deg(z)=4k-2$$
		Then the $S^1$-equivariant Poincar$\acute{e}$ series
		of $\Lambda_h M$ is given by
		\be
		\begin{split}
			P^{S^1}(\Lambda_h M; \mathbb{Q})(t) & =\frac{1-t^{4k}}{(1-t^2)(1-t^{4k-2})} \\
			& = \frac{1}{1-t^2}+\frac{t^{4k-2}}{1-t^{4k-2}}\\
			&= (1+t^2+t^4+\cdots+t^{2k}+\cdots)\\ &~~~~~~~~~~~~~~~+(t^{4k-2}+t^{2(4k-2)}+t^{3(4k-2)}+\cdots),
		\end{split}\nonumber
		\ee
		which yields Betti numbers
		\be
		b_q = \mathrm{rank} H_q^{S^1}(\Lambda_h M;\mathbb{Q})
		= \begin{cases}
			2,&\quad {\it if}\quad q\in \{2j(n-1)\mid j\in\mathbb{N}\}, \\
			1,&\quad {\it if}\quad q\in2\mathbb{N}_0\setminus \{2j(n-1)\mid j\in\mathbb{N}\}, \\
			0, &\quad {\it otherwise}.
		\end{cases}
		\ee
	\end{enumerate}
\end{lemma}

Consider the energy functional on a compact Finsler manifold $(M,F)$ defined as
\be \label{E:Energy}
E(\gamma) = \frac{1}{2} \int_{S^1} F^2(\gamma,\dot{\gamma}) dt. 
\ee
It is a $C^{1,1}$ functional on the free loop space $\Lambda M$ and satisfies the Palais-Smale condition. For any $\kappa \in \BFR$ we denote the sub-level set by  
\be
\Lambda^\kappa=:\{ d\in \Lambda_h M | \; E(d) \leq \frac{1}{2} \kappa^2 \;\}.
\ee
The sub-level set of a closed geodesic is defined as 
\be
\Lambda(c) = \{ d \in \Lambda_h M |\; E(d) \leq E(c) \}.
\ee 
Consider the energy functional $E$ on the non-contractible component $\Lambda_h M$. The $S^1$-critical modules of $c^{p(m-1)+1}$ for $E|_{\Lambda_h M}$ is defined as
$$ \overline{C}_*(E,c^{p(m-1)+1}; [h])
= H_*\left((\Lambda_h(c^{p(m-1)+1})\cup S^1\cdot c^{p(m-1)+1})/S^1,\Lambda_h(c^{p(m-1)+1})/S^1; \mathbb{Q}\right).$$
Following \cite{Rad92}, Section 6.2, we can use finite-dimensional approximations to $\Lambda_h M$ to apply the results of
D. Gromoll and W. Meyer \cite{GM1969Top} to a given non-contractible closed geodesic $c$ of class $[h]$ which is isolated as a critical orbit and obtain
\begin{proposition} \label{P:dim-1}
	Let $k_j(c)\equiv\dim\overline{C}_j(E,c; [h])$. Then $k_j(c)$  equal to $0$ when $j<i(c)$ or $j>i(c)+\nu(c)$. $k_j(c)$ is
	either $0$ or $1$ when $j = i(c)$ or $j=i(c) + \nu(c)$. If $\nu(c)=0$, i.e., $c$ is non-degenerate, then $k_j(c)=1$ holds only for $j=i(c)$.
\end{proposition}
Note that for a non-contractible minimal closed geodesic $c$ of class $[h]$, its $m$-th iterations $c^m\in\Lambda_h M$
if and only if $m\equiv 1(\mod~ p)$. Then we have the Morse inequality for non-contractible closed geodesics of class $[h]$:
\begin{lemma}[Theorem I.4.3 of \cite{Chan93}]
	Assume that $M=S^n/\Gamma$ be a Finsler manifold with finitely many non-contractible minimal closed geodesics of class $[h]$,
	denoted by $\{c_j\}_{1\le j\le k}$. Set
	\begin{align}
		M_q =\sum_{1\le j\le k,\; m\ge 1}\dim{\overline{C}}_q(E, c^{p(m-1)+1}_j; [h]),\quad q\in\mathbb{Z}.
	\end{align}
	Then for every integer $q \ge 0$ there holds
	\bea \label{E:Morse-in}
	M_q \ge {b}_q. \eea
\end{lemma}

\subsection{Index theory of symplectic patch and precise iteration formula}

Now we list some results on the maslov-type index theory of symplectic path, in particular the precise index iteration formula. For more details we refer to the monograph of Long \cite{lo2002} and references therein. Suppose $P$ is a symplectic matrix in Sp$(2n-2)$ and $\Omega^{0}(P)$ is the path-connected component of
its homotopy set $\Omega(P)$ which contains $P$.
Then there is a path $f\in C([0,1],\Omega^{0}(P))$
such that $f(0)=P$ and
\bea \label{E:path-1} f(1)
&=& N_1(1,1)^{\dm p_-}\,\dm\,I_{2p_0}\,\dm\,N_1(1,-1)^{\dm p_+}\nonumber\\
& &  \dm\,N_1(-1,1)^{\dm q_-}\,\dm\,(-I_{2q_0})\,\dm\,N_1(-1,-1)^{\dm q_+} \nonumber\\
&&\dm\,R(\th_1)\,\dm\,\cdots\,\dm\,R(\th_{r'})\,\dm\,R(\th_{r'+1})\,\dm\,\cdots\,\dm\,R(\th_r) \\
&&\dm\,N_2(e^{i\aa_{1}},A_{1})\,\dm\,\cdots\,\dm\,N_2(e^{i\aa_{r_{\ast}}},A_{r_{\ast}})\nonumber\\
& &  \dm\,N_2(e^{i\bb_{1}},B_{1})\,\dm\,\cdots\,\dm\,N_2(e^{i\bb_{r_{0}}},B_{r_{0}})\nonumber\\
& &\dm\,H(\pm 2)^{\dm h},\nonumber\eea
where $N_{1}(\lambda,\chi)=\left(
\begin{array}{ll}
	\lambda\quad \chi\\
	0\quad \lambda\\
\end{array}
\right)$ with $\lambda=\pm 1$ and $\chi=0,\ \pm 1$; $H(b)=\left(
\begin{array}{ll}
	b\quad 0\\
	0\quad b^{-1}\\
\end{array}
\right)$ with $b=\pm 2$;
$R(\theta)= \left(
\begin{array}{ll}
	\cos\theta\ -\sin\theta\\
	\sin\theta\quad\ \cos\theta\\
\end{array}
\right)$ with $\theta\in (0,2\pi)\setminus\{\pi\}$ and we suppose that $\pi<\theta_{j}<2\pi$ iff $1\leq j\leq
r'$;  $$N_{2}(e^{i\aa_{j}},A_{j})=\left(
\begin{array}{ll}
	R(\aa_{j})\  A_{j}\\
	\ 0\quad\ R(\aa_{j})\\
\end{array}
\right)\ \text{and}\ N_{2}(e^{i\bb_{j}},B_{j})=\left(\begin{array}{ll}
	R(\bb_{j})\  B_{j}\\
	\ 0\quad\ R(\bb_{j})\\
\end{array}
\right)$$ with $\alpha_{j},\beta_{j}\in (0,2\pi)\setminus\{\pi\}$ are non-trivial and trivial basic normal forms respectively.

Let $\gamma_{0}$ and $\gamma_{1}$ be two $\omega$-homotopic symplectic paths in Sp$(2n-2)$ connecting the identity matrix $I$ to $P$ and $f(1)$.
Then it has been shown that $i_{\omega}(\gamma_{0}^{m})=i_{\omega}(\gamma_{1}^{m})$ for any $\omega\in S^{1}=\{z\in\mathbb{C} \mid|z|=1\}.$
Based on this fact, we always assume without loss of generality that each $P_{c}$ appearing in the sequel has the form (\ref{E:path-1}).\\

Suppose $E(a)=\min\{k\in\Z\,|\,k\ge a\}$,
$[a]=\max\{k\in\Z\,|\,k\le a\}$, $\varphi(a)=E(a)-[a]$ and $\{a\}=a-[a]$ for any $a\in\BFR$.

For $m \in \mathbf{N}$, define  
$$E_{m}(a)=E\left(a-{1-(-1)^{m}\over4}\right),\ \varphi_{m}(a)=\varphi\left(a-{1-(-1)^{m}\over4}\right),\ \forall a\in\R.$$. We now have 
\begin{lemma}[Theorem 3.1 of \cite{LX}]
	\label{L:non-orient}
	Let $c$ be a non-orientable closed geodesic on a $n$-dimensional Finsler manifold with its linear
	Poincar$\acute{e}$ map $P_c$. Then for every $m\in\BFN$, we have
	\begin{equation} \label{E:ind-2}
		\begin{split}
			i(c^{m})& = m(i(c)+q_{0}+q_{+}-2r')-(q_{0}+q_{+})-{1+(-1)^{m}\over2}\left(r+p_{-}+p_{0}+{1-(-1)^{n}\over2}\right) \\
			&+2\sum_{j=1}^{r} E_{m}\left({m\theta_{j}\over2\pi}\right)+2\sum_{j=1}^{r_*} \varphi_{m}\left({m\alpha_{j}\over2\pi}\right)-2r_*, 
		\end{split}
	\end{equation}
	and
	\bea \label{E:nul-2}
	\nu (c^{m})&= \nu (c)+{1+(-1)^{m}\over2}\left(p_{-}+2p_{0}+p_{+}+{1-(-1)^{n}\over2}\right)+2\tilde{\varsigma}(c,m),\nonumber
	\eea
	where we denote by
	$$\tilde{\varsigma}(c,m)=\left(r-\sum_{j=1}^{r}\varphi_{m}\left({m\theta_{j}\over2\pi}\right)\right)
	+\left(r_{*}-\sum_{j=1}^{r_{*}}\varphi_{m}\left({m\alpha_{j}\over2\pi}\right)\right)
	+\left(r_{0}-\sum_{j=1}^{r_{0}}\varphi_{m}\left({m\beta_{j}\over2\pi}\right)\right).$$
\end{lemma}

\subsection{Geometric properties concern Finsler manifolds} 

The following results are significant in our considerations. 
\begin{lemma}[cf. Lemma 1 in \cite{Rad07}] \label{L:Length}
	Let $(M,F)$ be a compact and simply-connected Finsler manifold with reversibility $\lambda$ and flag curvature $K$ satisfying $0<K\leq 1$ resp. $\left( \frac{\lambda}{\lambda+1} \right)^2 < K \leq 1$ if the dimension is odd. Then the length of a closed geoddesic is bounded from below by $\pi \frac{\lambda+1}{\lambda}$. 
\end{lemma}

\begin{lemma}[cf. Theorem 4 and 5 in \cite{Rad04}] \label{L:index-c}
	Let $c$ be a closed geodesic on a Finsler manifold $(M,F)$ of dimension $n$ with positive flag curvature $K \geq \delta$ for some $\delta \in \BFR^+$.
	\begin{itemize}
		\item The mean average index is bounded from below: $\alpha(c) \geq \sqrt{\delta}\frac{n-1}{\pi}$. 
		\item If the length $L(c)$ satisfies $L(c) > \frac{k \pi}{\sqrt{\delta}}$ for some positive integer $k$ then $ind(c) \geq k(n-1)$.
	\end{itemize}
\end{lemma}

\begin{lemma} \label{L:mean-ind}
	Suppose $h \in \Gamma$ is non-trivial element order $p \geq 2$. Let $c$ be a closed geodesic on a compact Finsler space form $(S^n/\Gamma,F)$ of class $[h]$ with a non-reversible Fimsler metric with reversibility $\lambda$ and flag curvature $0<\delta \leq K \leq 1$ where $\delta>\frac{\lambda^2}{(\lambda+1)^2}$ if $n$ is odd.
	\[
	\hat{i}(c) \geq \sqrt{\delta} \frac{\lambda+1}{\lambda}\frac{n-1}{p}
	\]
\end{lemma} 

\begin{proof}
	Due to Lemma \ref{L:Length} we have  $L(c^p) \geq \frac{\lambda+1}{\lambda} \pi$ which follows $L(c) \geq \frac{\lambda+1}{\lambda} \frac{\pi}{p}$. Moreover, due to Lemma \ref{L:index-c}, we have  
	\[
	\hat{i}(c) \geq \sqrt{\delta}\frac{n-1}{\pi} L(c) \geq  \sqrt{\delta}\frac{n-1}{\pi} \frac{\lambda+1}{\lambda} \frac{\pi}{p} = \sqrt{\delta} \frac{\lambda+1}{\lambda}\frac{n-1}{p}.
	\]
	Then the proof is finished.	
\end{proof}

\section{Existence of non-contractible closed geodesics on $\mathbb{R}P^n$} \label{S:2-dim}

In this section, we focus on the real projective space $\mathbb{R}P^n$. Note that the $\mathbb{Z}_2$ is the only non-trivial finite group acting on $S^n$ freely and isometrically provided $n$ is even. It actually covers even-dimensional compact space forms.  We claim that there at least one non-contractible closed geodesic on the pinched real projective space without self-intersection. 

For the real project space $\R P^{n}$, the free loop space is decomposed as 
\[
\Lambda M = \Lambda_e M \sqcup \Lambda_h M.
\]
Clearly the minimizer of $E$ in $\Lambda_h M$ is a non-constant closed geodesic $c_1$ of class $[h]$ with $i(c_1)=0$.  By Lemma \ref{L:index-c}, we have $L(c_1) \leq \frac{\pi}{\sqrt{\delta}}$ since $i(c_1) <(n-1)$. Note that Lemma \ref{L:Length} also holds for geodesic loops. The shortest closed geodesic loops is bounded by $\frac{\pi}{2} \frac{\lambda+1}{\lambda}$. Therefore we have 
\be
L(c_1)\leq \frac{\pi}{\sqrt{\delta}}<2 \frac{\pi}{2} \frac{\lambda+1}{\lambda} = \frac{\lambda+1}{\lambda} \pi
\ee 
provided $\left(\frac{\lambda}{\lambda+1}\right)^2<\delta \leq K \leq 1$. Therefore $c_1$ is a non-contractible closed geodesic without self-intersections. Now let us consider the multiplicity of non-contractible closed geodesic on $\mathbb{R}P^2$ with length $L$ at most. 
\begin{proof}[Proof of Theorem \ref{T:main-1}]
	By letting $p=2$ in Lemma \ref{L:mean-ind} one has 
	\be 
	\hat{i}(c) \geq \sqrt{\delta} \frac{\lambda+1}{2 \lambda}. 
	\ee
We claim that there must be another closed geodesic on $(\mathbb{R}P^2,F)$ with $\left(\frac{\lambda}{\lambda+1}\right)^2<\delta \leq K<1$. 
	Indeed, suppose $N$ is an integer such that 
	\[
	N-1 \leq \frac{1}{\sqrt{\delta} \frac{\lambda+1}{\lambda} -1 }<N.
	\]
	(1) Suppose integer $N$ is even. \\
	
	Assume that there is only one minimal closed geodesics in $\Lambda_h M$. Then any other closed geodesics $\tilde{c}$ with $i(\tilde{c}) \leq N$ must be the iterations of $c_1$, i.e., $\tilde{c}=c_1^{2(m-1)+1}$ for some $m \in \BBN$. By Lemma \ref{L:Betti-Tai}, we have the Betti numbers for $n=2$ 
	\[
	b_q= \begin{cases} 2, \quad & q = 2j,\; j \in \BBN, \\
		1, \quad &  q=0, \\
		0,\quad & \text{ otherwise },
	\end{cases}
	\]
	and the Morse inequality \eqref{E:Morse-in} follows  
	\be \label{E:Morse-inq}
	\sum_{0 \leq q \leq N,\; \text{$q$ is even}} M_q \geq \sum_{0 \leq q \leq N,\; \text{$q$ is even}} b_q = 1+ 2 \times \frac{N}{2} = N+1.
	\ee
	On the other hand, the precise iteration formula of non-orientiable closed geodesics \eqref{E:ind-2} implies 
	\[
	\begin{split}
		i(c_1^{2(m-1)+1}) & = 2(m-1)(q_0+q_+) + 2 \sum_{j=1}^r E_m\left(\frac{(2(m-1)+1)\theta_j}{2\pi}\right) \\
		& + 2 \sum_{j=1}^{r_*} \varphi_m\left(\frac{(2(m-1)+1)\alpha_j}{2\pi}\right) - 2 r^* \equiv 0 \;\; \big({\rm mod \;} 2\big).
	\end{split}
	\]
	It implies that $i(c_1^{2(m-1)+1})$ is always even and $\nu(c_1^{2(m-1)+1}) = \nu(c_1) \;({\rm mod} \; 2)$ due to \eqref{E:nul-2} for every $m \in \mathbb{N}$. Recall Proposition \ref{P:dim-1} which indicates $k_q(c) =0$ for $q <i(c), \; q >i(c) + \nu(c)$ and $k_{i(c)} + k_{i(c) + \nu(c)} \leq 1$. We have 
	\[
	\sum_{\text{$q$ is even}} k_q(c_1^{2(m-1)+1}) \leq 1.
	\]
	for each $m \geq 1$ since  $\nu(c_1^{2(m-1)+1}) \leq 2$. The following inequality holds
	\[
	\sum_{0 \leq q \leq N, \; \text{$q$ is even}} M_q = \sum_{0 \leq q \leq N, m \geq 1,\; \text{$q$ is even}} k_q(c_1^{2(m-1)+1}) \leq \sharp \left\{ m \;\big| i(c_1^{2(m-1)+1}) \leq N \right\}. 
	\]
On the other hand, due to the mean index inequality, we have 
	\[
	\begin{split}
	i(c_1^{2(m-1)+1}) &\geq (2(m-1)+1)\hat{i}(c_1) -1 \geq (2(m-1)+1) \left(\sqrt{\delta}\frac{\lambda+1}{2\lambda} -\frac{1}{2}\right)  + \frac{2(m-1)+1}{2} - 1 \\
	&> m - \frac{3}{2} + \frac{2(m-1)+1}{2N}.
	\end{split}
	\]
Suppose that $m - \frac{3}{2} + \frac{2(m-1)+1}{2N} >N$. It follows $m \geq N+1$ then implies 
	\[
	\sum_{0 \leq q \leq N, \; \text{$q$ is even}} M_q\leq \sharp \left\{ m \;\big| i(c_1^{2(m-1)+1}) \leq N \right\} \leq N.
	\]
It contradicts with \eqref{E:Morse-inq}. Therefore there must be another prime non-contractible closed geodesics $c_2$ of class $[h]$ satisfying $i(c_2) \leq N$. Due to Lemma \ref{L:index-c} we have 
	\[
	L(c_2) \leq \frac{N+1}{\sqrt{\delta}} \pi \leq  \frac{\pi}{\sqrt{\delta}} \left(2 + \frac{1}{\sqrt{\delta}\frac{\lambda+1}{\lambda}-1}\right).
	\]

(2) Suppose $N$ is an odd integer. We claim that there must be another non-contractible closed geodesic with $i(c_2) \leq N-1$. Otherwise, assuming there is only one non-contractible minimal closed geodesic, any other closed geodesic $\wt c$ with $i(\wt c) \leq N-1$ must be the iteration of $c_1$. Then the Morse inequality \eqref{E:Morse-inq} follows 
\[
\sum_{0 \leq q \leq N-1,\; \text{$q$ is even}} M_q \geq \sum_{0 \leq q \leq N-1,\; \text{$q$ is even}} b_q = 1 + 2 \times \frac{N-1}{2} =N.
\] 
Repeating the proof for even case by insteading $N$ by $N-1$, we obtain that there must be another prime closed geodesic $c_2$ with $i(c_2) \leq N-1$. Then due to Lemma \ref{L:index-c}, one has 
\[
L(c_2) \leq \frac{N}{\sqrt{\delta}} \pi \leq \frac{1}{\sqrt{\delta}} \left(1+ \frac{1}{\sqrt{\delta}\frac{\lambda+1}{\lambda}-1} \right) \pi.
\]
\end{proof}

\section{Equivariant Morse theory for non-contractible closed geodesics} \label{S:Equivariant}
 
However, if the metrics are assumed to be not far from standard Riemannian metrics, we could obtain more simple closed geodesics. 

\begin{lemma} \label{L:homology}
There are $n$ sub-ordinate non-zero homology classes $\sigma_k \in H^{S^1}_{2k-2}(\Lambda_h M)$ for $1 \leq k \leq n$.
\end{lemma}

\begin{proof}
Let $g_0$ be the standard metric on $M$. The corresponding energy functional is given by
\[
E_0(c) = \frac{1}{2} \int_{S^1} g_0(\dot{c}(t),\dot{c}(t)) dt.
\]
All prime non-contractible closed geodesics on $(S^n/\Gamma,g_0)$ of class $[h]$ are of length $\frac{2\pi}{p}$ and they are completely degenerate, i.e., all the eignvalues of the Poincar\'{e} map is $1$. If $c$ is a closed geodesic of class $[h]$ then $(m-1)p+1$-th iteration  $c^{(m-1)p+1}$ is for $m \geq 2$.   

The critical manifolds consist of $c^{(m-1)p+1}$ is diffeomorphic to the unit tangential  bundle $STM:=\{(x,v) \in TM \big| F(x,v)=1 \}$. The index is 
\[
\left(p(m-1)+1 -1 \right)(n-1) = p(m-1)(n-1)
\]
The $S^1$-equivariant homology group of $\Lambda_h M$ is generated by the local critical groups
\be
\begin{split}
 H^{S^1}_*(\Lambda_h M) & \simeq \bigoplus_{m \geq 1}	H^{S^1}_{*} (\Lambda_{0,h}^{\frac{2\pi^2(p(m-1)+1)^2}{p^2} + \ep} M, \; \Lambda_{0,h}^{\frac{2 \pi^2 (p(m-1)+1)^2}{p^2} - \ep} M) \\
	& \simeq  \bigoplus_{m \geq 1} H^{S^1}_{*-p(m-1)(n-1)} (STM),  
\end{split}
\ee 
where $\Lambda_{0,h}^\kappa= \{ d \in \Lambda_h M | E_0(d) \leq \kappa \}$ and $\ep>0$ is a sufficiently small number. The second isomorphism comes from the fact 
\[
H^{S^1}_*(M^b,M^a) \simeq \bigoplus H^{S^1}_{*-\lambda_i}(A_i),
\]
where $c\in (a,b)$ is a critical value of Morse function $f:M\rightarrow \R$, $K \cap f^{-1}(c) = \bigcup_{i \in I} A_i$ is the critical manifolds and $\lambda_i$ is the index of $A_i$. The isotropy group of the $S^1$-action on $STM$ is $Z_m$ for $m \in \mathbf{N}$. Due to \cite{FR78} and the universal coefficient theorem, one has
\be
H_{*}^{S^1} (STM) \simeq H_*(STM / S^1).
\ee
Now let $m=1$. Then the $S^1$-action on $STM$ is a free group action hence $STM/S^1$ is a smooth manifold. We have the $S^1$-fibration $S^1 \rightarrow STM \rightarrow STM/S^1$. Let $e \in H^2(STM/S^1)$ be the Euler class of the $S^1$-fibration. Hence $e^{n-1} \neq 0$ is a generator of $H^{2n-2}(STM/S^1)$. 

Denote by $[C]$ the fundamental homology class of $STM/S^1$. We have $n$ sub-ordinate nonzero homology class $\alpha_k \in H^{S^1}_{2k-2} (STM/S^1)$ for $1 \leq k \leq n$ denoted by $\alpha_k =[C] \cap e^{n-k}$ in which $\cap$ is the cap product. Note that $STM/S^1$ is a non-degenerate $S^1$-manifold. It follows from the handle-bundle theorem that $\Lambda_{0,h}^{\frac{2\pi^2(p(m-1)+1)^2}{p^2}+\ep}$ is $S^1$-homotopic to $\Lambda_{0,h}^{\frac{2\pi^2(p(m-1)+1)^2}{p^2}-\ep}$ with the handle-bundle $DN^{-}$ attached along $SN^{-}=\p DN^{-}$ where $DN^{-}$ is the closed disk bundle of the negative bundle over $STM/S^1$. In particular, ${\rm rank} DN^{-} =0$. Now we have 
\be
H^{S^1}_{*}(\Lambda_{0,h}^{\frac{2 \pi^2}{p^2}+\ep}, \Lambda_{0,h}^{\frac{2 \pi^2}{p^2}-\ep}) \simeq H^{S^1}_{*}(DN^{-},SN^{-}) \simeq H^{S^1}_{*}(STM), 
\ee
where the last isomorphism is given by the Thom isomorphism $\Phi$. Let $f:(\Lambda_h M)_{S^1} = \Lambda_h M \times_{S^1} S^\infty \rightarrow \mathbf{C} P^\infty$ be a classifying map and $\eta \in H^2 (\mathbf{C} P^\infty)$ be the universal first rational Chern class. Let $\sigma_k = \Phi^{-1}(\alpha_k) = \Phi^{-1}([C]\cap e^{n-k}) \neq 0$. Since the Euler class $e \in H^2(STM)$ coincide with the first Chern class $c_1 \in H^2(STM)$ of the $S^1$-bundle $STM \rightarrow STM/S^1$, $DN^{-}$ is $S^1$-homotopic to $STM$ and the $S^1$-acion is free on $STM$ thus $(STM)_{S^1}$ is homotopic to $STM/S^1$, we have $\sigma_k = \Phi^{-1}([C]) \cap f^*(\eta)^{n-k}\in H_{2k-2}^{S^1}(\Lambda_h M)$.   

\end{proof}

Suppose that $\sigma_k \in H^{S^1}_{2k-2}(\Lambda_h M)$ for $1 \leq k \leq n$. Let 
\be \label{E:minimax}
\lambda_k := \inf_{\gamma \in \sigma_k} \sup_{d \in {\rm im} \gamma} E(d).
\ee 
Then we have the following proposition.
\begin{proposition} \label{P:critical-value}
Each $\lambda_k$ is a critical value of $E$ and $0< \lambda_1 \leq  \lambda_2 \leq \cdots \leq \lambda_{n}$. In particula, if $\lambda_k=\lambda_{k+1}$ for some $1\leq k \leq n-1$, then there are infinitely many prime closed geodesics on $(M,F)$.  
\end{proposition}

\begin{proof}
Since $E$ is $C^{1,1}$ on $\Lambda_h M$. $\lambda_1$ is the minimal point of $E$ in $\Lambda_h M$ hence it is a critical value. Assume that $\lambda_k$ is not a critical value of $E$ for some $2 \leq k \leq n$. Due to the P-S condition, it is standard to construct an $S^1$-equivariant negative gradient flow such that we can push a chain $\gamma \in \sigma_k$ with $\sup_{d \in Im \gamma} E(d) < \lambda_k + \rho$ down the level $\lambda_k - \rho$. It contradicts with the definition of $\lambda_k$.  Therefore, $0<\lambda_1\leq \lambda_2 \leq \cdots \leq \lambda_{n}$ are critical values of $E$ where the order is due to the definition of cap product. 

If $\lambda_k= \lambda_{k+1}$ for some $1 \leq k \leq n-1$, a similar argument as Proposition 2.2 of \cite{Wan12} implies that there would be infinitely many non-contractible closed geodesics of class $[h]$.  The proof is complete.

\end{proof}

\begin{lemma} \label{L:Critical-values}
	Suppose that there exists $\delta_k>0$ such that any closed geodesic $c$ with $E(c) \in (\lambda_k-\delta_k,\lambda_k+\delta_k)$ is isolated. Then there exists a clsoed geodesic $c_k$ such that 
	\be
	E(c_k) = \lambda_k,\;\; i(c_k) \leq 2k-2 \leq i(c_k) + \nu(c_k)
\ee for $1 \leq k \leq n$.
\end{lemma}

\begin{proof}
By the isolatedness of critical values and Palais-Smale condition, there exists a small number $\ep >0$ such that $\lambda_k$ is the unique critical value of $E$ in $(\lambda_k-\ep,\lambda_k+\ep)$. Due to the definition of $\lambda_k$ in \eqref{E:minimax}, we have a chain $\gamma \in \sigma_k$ such that $\sup_{d \in Im \gamma} E(d) < \lambda_k + \ep$. Due to Lemma \ref{L:homology}, we have 
\[
0 \neq [\gamma] \in H^{S^1}_{2k-2}(\Lambda_h^{\lambda_k+\ep}).
\]
Assume that $H^{S^1}_{2k-2}(\Lambda_h^{\lambda_k+\ep}, \; \Lambda_h^{\lambda_k-\ep}) = 0$. Due to the exactness of homology sequence 
\be
H_{2k-2}^{S^1} (\Lambda_h^{\lambda_k-\ep}) \xrightarrow{i_*}  H_{2k-2}^{S^1} (\Lambda_h^{\lambda_k+\ep}) \xrightarrow{j_*} H^{S^1}_{2k-2}(\Lambda_h^{\lambda_k+\ep}, \; \Lambda_h^{\lambda_k-\ep}),
\ee
where $i: \Lambda_h^{\lambda_k-\ep} \rightarrow \Lambda_h^{\lambda_k+\ep}$ and $j:\left(\Lambda_h^{\lambda_k+\ep},\emptyset \right) \rightarrow (\Lambda_h^{\lambda_k+\ep}, \; \Lambda_h^{\lambda_k-\ep})$,
we have  ${\rm im}\; i_* =  H_{2k-2}^{S^1} (\Lambda_h^{\lambda_k+\ep})$ thus we have $\gamma^\prime \in  H_{2k-2}^{S^1} (\Lambda_h^{\lambda_k-\ep})$ such that $\gamma = i_*(\gamma^\prime)$. It  contradicts with the definition of $\lambda_k$. 

Since $\lambda_k$ is the unique critical value of $E$ in $(\lambda_k-\ep,\lambda_k+\ep)$. We have 
\be
H^{S^1}_*(\Lambda_h^{\lambda_k+\ep},\Lambda_h^{\lambda_k-\ep}) \simeq \bigoplus_{c \in Crit(E)} Crit^{S^1}_*(E,S^1\cdot c) = \bigoplus_{c \in Crit(E)} H^{S^1}_*\left(\Lambda(c) \cap N, (\Lambda(c) \setminus S^1\cdot c) \cap N \right),
\ee 
where $N$ is an $S^1$-invariant open neighbourhood of $S^1\cdot c$ such that $Crit(E)\cap N = S^1\cdot c$. Then by Lemma 6.11 of \cite{FR78}, we have 
\[
H^{S^1}_*\left(\Lambda(c) \cap N, (\Lambda(c) \setminus S^1\cdot c) \cap N \right) \simeq H_*\left((\Lambda(c) \cap N)/S^1, ((\Lambda(c) \setminus S^1\cdot c) \cap N)/S^1 \right).
\]
Then consider thre finite-dimensional approximating to $\Lambda_h M$ and apply the Gromoll-Meyer theory, we have 
\[
H_q\left((\Lambda(c) \cap N)/S^1, ((\Lambda(c) \setminus S^1\cdot c) \cap N)/S^1 \right)=0,
\]
provided $q>i(c)+\nu(c)$ or $q<i(c)$. The proof is completed.
\end{proof}

Now we are in the position to prove Theorem \ref{T:main-2} and Theorem \ref{T:main-3}.

\begin{proof}[Proof of Theorem \ref{T:main-2}] 

Suppose $F^2 <(\frac{\lambda+1}{\lambda})^2 g_0$ and $l(M,F)  \geq \frac{\pi}{p}(\frac{\lambda+1}{\lambda})$. For sufficiently small $\ep>0$, one obtains 
$$
\Lambda_{0,h}^{\frac{2 \pi^2}{p^2} + \ep} M \subset \Lambda_h^{ \frac{2 \pi^2}{p^2} (\frac{\lambda+1}{\lambda})^2-\ep} M.
$$ 
Due to Proposition \ref{P:critical-value}, there exists $n$ critical values of $E$
\be \label{E:critical-value}
0 < \lambda_1 \leq \lambda_2 \leq \cdots \leq \lambda_{n} < \frac{2 \pi^2}{p^2} (\frac{\lambda+1}{\lambda})^2.
\ee
Therefore, we have $n$ distinct closed geodesics $\{c_k\}_{k=1}^{n}$ on $(S^n/\Gamma,F)$ which are homotopic to $[h]$ with $E(c_k)= \lambda_k$. We claim each $c_k$ is a minimal closed geodesic without self-intersections. Otherwise, one would obtain 
\[
L(c_k) \geq 2 \frac{\pi}{p} \frac{\lambda+1}{\lambda}
\]
since any closed geodesic loop satisfies $l(M,F) \geq \frac{\pi}{p} \frac{\lambda+1}{\lambda}$. It follows that  
\be
E(c_k) \geq \frac{1}{2} L(c_k)^2 \geq 2 \frac{\pi^2}{p^2} (\frac{\lambda+1}{\lambda})^2,
\ee 
which is contradicted with condition \eqref{E:critical-value}. 

When $\lambda_1<\lambda_2<\cdots<\lambda_{n}$, it is clearly $n$ non-contractible closed geodesics. On the other hand, if there exists $1 \leq k \leq n-1$ such that $\lambda_k = \lambda_{k+1}$, there are infinitely many non-contractible closed geodesics of class $[h]$ below the energy $\lambda_n+\ep$ which are not self-intersected due to \eqref{E:critical-value}. The proof is now complete.
\end{proof}
	
In the next we consider the stability of these non-contractible closed geodesics.

\begin{proof}[Proof of Theorem \ref{T:main-3}]
	
It is sufficient for us to focus on the case  $0<\lambda_1 <\lambda_2<\cdots <\lambda_{n}$ are isolated cirtical values since another cases are similar with \cite{Wan12}.

%Assume there exists some $1\leq k \leq n-1$ such that $\lambda_k = \lambda_{k+1}$ in Proposition \ref{P:critical-value}. Similar with \cite{Wan12}, there must be infinitely many non-hyperbolic closed geodesics of class $[h]$ on $(M,F)$.  If there does not exist $\delta_k>0$ such that any closed geodesic $c$ with $E(c) \in (\lambda_k-\delta_k,\lambda_k+\delta_k)$ is isolated. Then we can find a degenerate closed geodesic with $E(c_k)=\lambda_k$,  which is clearly non-hyperbolic ($1$ is a Poincar\'e multiplier). On the other hand,

Applying Lemma \ref{L:Critical-values} we can find a closed geodesic $c_k$ such that $E(c_k)=\lambda_k$ and $i(c_k) \leq 2 k-2 \leq i(c_k) + \nu(c_k)$. Assume that $c_k$ is hyperbolic, then we have $i(c_k) =2k-2$. Recall %the pinched condition $\left(\frac{\lambda}{\lambda+1} \right)^2 \leq \delta < K \leq 1$ implies 
\[
L(c_k) = \frac{1}{p} L(c_k^p) \geq \frac{\lambda +1}{\lambda} \frac{\pi}{p}.
\]
Let us choose $m,l \in \BFN$ such that 
\[
\frac{\lambda+1}{\lambda} \sqrt{\delta} - \ep_0 <\frac{pl}{p(m-1)+1} < \frac{\lambda+1}{\lambda} \sqrt{\delta}
\]
for some small $\ep_0>0$. We have 
\[
L(c_k^{p(m-1)+1}) \geq (p(m-1)+1) \frac{\pi}{p} \frac{\lambda+1}{\lambda} > l \frac{\pi}{\sqrt{\delta}}.
\] 
Lemma \ref{L:index-c} implies
\be
i(c_k^{p(m-1)+1}) \geq l (n-1).
\ee
On the other hand, if closed geodesics are hyperbolic, we must have $i(c_k^{p(m-1)+1}) = 2(k-1)(p(m-1)+1)$ and obtain
\[
\frac{pl}{p(m-1)+1} \leq \frac{2(k-1)p}{n-1}.
\]  
Therefore $c_k$ is a non-hyperbolic closed geodesic for $k \geq 1$ satisfies 
\[
\frac{2(k-1)p}{n-1} < \frac{\lambda+1}{\lambda} \sqrt{\delta}. 
\]
We have at least $\left[ \frac{\lambda+1}{\lambda} \sqrt{\delta} \frac{n-1}{2p}+1\right]$ non-hyperbolic non-contractible closed geodesics of class $[h]$.

On the other hand, suppose $F^2 <\frac{1}{\rho}g_0$ with $\rho >\left(\frac{\lambda}{\lambda+1}\right)^2$. We have $L(c_k) <\frac{2 \pi}{p} \frac{1}{\sqrt{\rho}}$. Now let 
$m,l \in \BFN$ such that 
\[
L(c_k)< \frac{l \pi}{p(m-1)+1}< \frac{2 \pi}{p} \frac{1}{\sqrt{\rho}}.
\]
Then we have $L(c_k^{p(m-1)+1}) < l \pi$ hence $i(c_k^{p(m-1)+1}) \leq l(n-1)$. Assume $c_k$ is hyperbolic, we have $i(c_k^{p(m-1)+1}) = 2(k-1)\left(p(m-1)+1\right)$ and the following inequality holds
\[
\frac{2k-2}{n-1} \leq \frac{l}{p(m-1)+1}
\] 
Therefore $c_k$ must be non-hyperbolic closed geodesics when
\[
k>\frac{n-1}{p} \frac{1}{\sqrt{\rho}}+1,
\]
In conclusion, we have 
\[
n - \left[ \frac{n-1}{p} \frac{1}{\sqrt{\rho}}+1 \right] + \left[\frac{n-1}{p} \frac{\lambda+1}{2\lambda} \sqrt{\delta}+1\right]
\]
non-hyperbolic non-contractible closed geodesics.

\end{proof}

{\bf Acknowledgement:} The author would like to thank Prof. Hui Liu for suggesting the issues on the non-contractible closed geodesics.


\begin{thebibliography}{99}
	\bibitem{Ano} D. V. Anosov, Geodesics in Finsler geometry. {\it Proc.
		I.C.M.} (Vancouver, B.C. 1974), Vol. 2. 293-297 Montreal (1975)
	(Russian), {\it Amer. Math. Soc. Transl.} 109 (1977) 81-85.
	\bibitem{Ano1980} D. V. Anosov, Some homotopies in spaces of closed curves. {\it Izv. Akad. Nauk SSSR Ser. Mat. } 44 (1980) 1219-1254 (Russian)
	%\bibitem{Apostol} Tom M. Apostol, Modular functions and Dirichlet series in Number theory, Graduate Texts in Mathematics 41, 1990.
	%\bibitem{BM2013} K. Burns, S. Matveev, Open problems and questions about closed geodesics, arXiv:1308.5417v2, 2014.
	\bibitem{Ban93} V. Bangert,  On the existence of closed geodesics on two-spheres.
	{\it Internat. J. Math.} 4 (1993), no. 1, 1--10.
	\bibitem{BH84} V. Bangert, N. Hingston, Closed geodesics on manifolds with infinite abelian fundamental group. {\it J.
		Differ. Geom.} 19 (1984), 277-282.
	\bibitem{BK1983} V. Bangert, W. Klingenberg, Homology generated by iterated closed geodesics. {\it Topology.}
	22 (1983), 379-388.
	\bibitem{BL10} V. Bangert, Y. Long, The existence of two closed geodesics on every Finsler 2-sphere,
	{\it Math. Ann.} 346 (2010), 335-366.
	%\bibitem{BCS00} D. Bao, S.S. Chern and Z. Shen, An Introduction to Riemannian-Finsler Geometry. Graduate Texts in Mathematics 200, Springer-Verlag. 2000
	\bibitem{BTZ81} W. Ballmann, G. Thorbergsson and W. Ziller, Closed geodesics and the fundamental group. {\it Duke Math. J.}
	48 (1981), 585-588.
	\bibitem{BTZ82} W. Ballmann, G. Thorbergsson and W. Ziller, Existence of closed geodesics on positively curved manifolds.
	{\it J. Differ. Geom.} 18 (1983), no. 2, 221-252
	%\bibitem{Bott1956} R. Bott,  On the iteration of closed geodesics and the Sturm intersection theory.  {\it Comm. Pure Appl. Math.} 9 (1956) 171-206.
	\bibitem{Bir27} G. Birkhoff, Dynamical Systems. Colloquium publications of American Mathematical Society. (1927)
	\bibitem{Chan93} K. C. Chang, Infinite Dimensional Morse Theory and Multiple Solution Problems. Birkh\"auser. Boston. 1993.
	
	%\bibitem{GHP19} D. Cristofaro-Gardiner, M. Hutchings and D. Pomerleano, Torsion contact forms in three dimensions have two or infinitely many Reeb orbits. {\it Geometry and Topology} 23 (2019), 3601-3645.
	
	\bibitem{DuL1} H. Duan, Y. Long, Multiple closed geodesics on
	bumpy Finsler $n$-spheres, {\it J. Differ. Equa.} 233 (2007), no. 1, 221-240.
	\bibitem{DuL2} H. Duan, Y. Long, The index growth and multiplicity of closed geodesics. {\it J. Funct.
		Anal.} 259 (2010) 1850-1913.
	\bibitem{DLW1} H. Duan, Y. Long and W. Wang, Two closed geodesics on compact simply-connected bumpy
	Finsler manifolds. {\it J. Differ. Geom.} 104 (2016), no. 2, 275-289.
	\bibitem{DLW2} H. Duan, Y. Long and W. Wang, The enhanced common index jump theorem for symplectic paths and
	non-hyperbolic closed geodesics on Finsler manifolds. {\it Calc. Var. and PDEs.} 55 (2016), no. 6, 55:145.
	\bibitem{DLX15} H. Duan, Y. Long and Y. Xiao, Two closed geodesics on $\mathbb{R}P^n$ with a bumpy Finsler metric,
	{\it Calc. Var. and PDEs}, (2015), vol 54,  2883-2894.
	%\bibitem{Eke90} I. Ekeland, Convexity Methods in Hamiltonian Mechanics. Springer-Verlag. Berlin. 1990.
	%\bibitem{EH87} I. Ekeland, H. Hofer,  Convex Hamiltonian energy surfaces and their closed trajectories,
	%{\it Comm. Math. Phys.} 113 (1987), 419-467.
	%\bibitem{EL80} I. Ekeland and J. Lasry, On the number of periodic trajectories for a Hamiltonian flow on a
	%convex energy surface. {\it Ann. of Math.} 112 (1980), 283-319.
	\bibitem{FR78} E. Fadell, P. Rabinowitz, Generalized cohomological index
	theories for Lie group actions with an application to bifurcation questions for
	Hamiltonian systems. {\it Invent. Math.} 45 (1978), no. 2, 139--174.
	
	\bibitem{Fra92} J. Franks, Geodesics on $S\sp 2$ and periodic points of annulus homeomorphisms. {\it Invent. Math.} 108 (1992), no. 2, 403-418.
	\bibitem{GM1969Top} D. Gromoll and W. Meyer, On differentiable functions with isolated critical points. {\it Topology}
	8 (1969), 361-369.
	\bibitem{GM69} D. Gromoll, W. Meyer, Periodic geodesics on compact Riemannian manifolds, {\it J. Differ. Geom.}
	3 (1969), 493-510.
	\bibitem{Hin84} N. Hingston, Equivariant Morse theory and closed geodesics, {\it J. Differ. Geom.} 19 (1984), 85-116.
	\bibitem{Hin93} N. Hingston,  On the growth of the number of closed geodesics on the two-sphere. {\it Inter. Math. Research Notices.} 9 (1993), 253-262.
	\bibitem{HiR} N. Hingston, H.-B. Rademacher, Resonance for loop homology of spheres. {\it J. Differ. Geom.}
	93 (2013), 133-174.
	%\bibitem{HWZ03} H. Hofer, K. Wysocki and E. Zehnder, Finite energy foliations of tight three-spheres and Hamiltonian dynamics. {\it Ann. of Math. } 157 (2003), 125-255.
	
	
	\bibitem{Kat73}  A. B. Katok, Ergodic properties of degenerate integrable Hamiltonian systems. {\it Izv. Akad. Nauk.
		SSSR} 37 (1973), [Russian]; {\it Math. USSR-Izv.} 7 (1973), 535-571.
	%\bibitem{KT72} W. Klingenberg, F. Takens, Generic properties of geodesic flows. {\it Math. Ann.} 197 (1972), 323-334.
	
	\bibitem{Kli78}  W. Klingenberg, Lectures on closed geodesics. Springer-Verlag, Berlin, heidelberg, New York, 1978.
	%\bibitem{Kli1995}  W. Klingenberg, Riemannian geometry.  de Gruyter; 2nd Rev ed. edition, 1995.
	\bibitem{Liu2005}  C. Liu, The Relation of the Morse Index of Closed Geodesics with the Maslov-type Index of Symplectic Paths,
	{\it Acta Math. Sinica } 21 (2005), 237-248.
	\bibitem{LL2002} C. Liu and Y. Long,  Iterated index formulae for closed
	geodesics with applications. {\it Science in China.} 45 (2002) 9-28.
	\bibitem{Liu17} H. Liu, The Fadell-Rabinowitz index and multiplicity of non-contractible closed geodesics on Finsler $\mathbb{R}P^{n}$.
	{\it J. Differ. Equa.} 262 (2017), 2540-2553.
	\bibitem{LLX18} H. Liu, Y. Long and Y. Xiao, The existence of two non-contractible closed geodesics on
	every bumpy Finsler compact space form. {\it Discrete Contin. Dyn. Syst.} 38 (2018), 3803-3829.
	\bibitem{LW22} H.Liu and Y. Wang, Multiplicity of non-contractible closed geodesics on Finsler compact space forms. {\it Calc. Var.} 61, 224 (2022)
	\bibitem{LW2023} H. Liu and Y. Wang, Generic existence of infinitely many non-contractible closed geodesics on compact space forms. {\it Acta Mathematica Sinica, English Series.} Accepted, 2023
	\bibitem{LX} H. Liu and Y. Xiao, Resonance identity and multiplicity of non-contractible closed geodesics on Finsler $\mathbb{R}P^{n}$.
	{\it  Adv. Math.} 318 (2017), 158-190.
	
	%\bibitem{Liu19} H. Liu, The optimal lower bound estimation of the number of closed geodesics on Finsler compact space form ${S}^{2n+1}/ \Gamma$. {\it Calc. Var. and PDEs}, 58 (2019), 107.
	\bibitem{lo1999} Y. Long,  Bott formula of the Maslov-type index theory.
	{\it Pacific J. Math.} 187 (1999), 113-149.
	\bibitem{lo2000} Y. Long, Precise iteration formulae of the Maslov-type index theory and
	ellipticity of closed characteristics. {\it Adv.  Math.} 154 (2000), 76-131.
	\bibitem{lo2002} Y. Long,  Index Theory for Symplectic Paths with Applications. Progress in Math. 207, Birkh\"auser. 2002.
	\bibitem{lo2006}  Y. Long, Multiplicity and stability of closed geodesics on Finsler 2-spheres, {\it J. Eur. Math. Soc.}
	8 (2006), 341-353.
	\bibitem{LoD} Y. Long, H. Duan,  Multiple closed geodesics on 3-spheres.
	{\it Adv. Math.} 221 (2009) 1757-1803.
	\bibitem{LW07}  Y. Long, W. Wang, Multiple closed geodesics on Riemannian 3-spheres, {\it Calc. Var. and PDEs},
	30 (2007), 183-214.
	\bibitem{LZ02} Y. Long, C. Zhu,  Closed characteristics on
	compact convex hypersurfaces in $\R^{2n}$.  {\it Ann. of Math.}
	155 (2002) 317-368.
	\bibitem{LF51} L. A. Lyusternik and A. I. Fet, Variational problems on
	closed manifolds. {\it Dokl. Akad. Nauk SSSR (N.S.)} 81 (1951) 17-18 (in Russian).
	
	%\bibitem{Oancea2014} A. Oancea, Morse theory, closed geodesics, and the homology of free loop spaces,With an appendix by Umberto Hryniewicz. {\it IRMA Lect. Math. Theor. Phys.}, 24,
	%Free loop spaces in geometry and topology, 67-109, Eur. Math. Soc., Z$\ddot{u}$rich, 2015. arXiv:1406.3107, 2014.
	\bibitem{Rad89} H.-B. Rademacher, On the average indices of closed geodesics, {\it J. Diff. Geom.} 29 (1989), 65-83.
	\bibitem{Rad92} H.-B. Rademacher, Morse Theorie und geschlossene Geodatische. {\it Bonner Math. Schr.} 229 (1992).
	%\bibitem{Rad94} H.-B. Rademacher, The Fadell-Rabinowitz index and
	%closed geodesics. {\it J. London. Math. Soc.} 50 (1994) 609-624.
	\bibitem{Rad04} H.-B. Rademacher, A Sphere Theorem for non-reversible
	Finsler metrics. {\it Math. Ann.} 328 (2004) 373-387.
	\bibitem{Rad07} H.-B. Rademacher, Existence of closed geodesics on
	positively curved Finsler manifolds. {\it Ergod. Th. Dyn. Sys. } 27 (2007), no. 3, 957--969.
	\bibitem{Rad08} H.-B. Rademacher, The second closed geodesic on the complex projective plane. {\it Front.
		Math. China.} 3 (2008), 253-258.
	\bibitem{Rad10} H.-B. Rademacher, The second closed geodesic on Finsler
	spheres of dimension $n>2$. {\it Trans. Amer. Math. Soc. } 362 (2010), no. 3, 1413-1421.
	\bibitem{RT20} H.-B. Rademacher, I. Taimanov, The second closed geodesic, the fundamental group, and generic Finsler metrics. arXiv:2011.01909v2, 2020
	\bibitem{Shen01} Z. Shen, Lectures on Finsler Geometry. World Scientific. Singapore. 2001.
	\bibitem{Tai1} I.A. Taimanov, The type numbers of closed geodesics. {\it Regul. Chaotic Dyn.} 15 (2010), no. 1, 84-100.
	\bibitem{Tai16} I.A. Taimanov, The spaces of non-contractible closed curves in compact space forms, {\it Mat. Sb.} 207(10) (2016), 105-118.
	\bibitem{ViS} M. Vigu$\acute{e}$-Poirrier, D. Sullivan, The homology theory of the closed geodesic problem. {\it J. Differ. Geom.}
	11 (1976), 633-644.
	\bibitem{Wan1} W. Wang, Closed geodesics on positively curved Finsler spheres.
	{\it Adv. Math.} 218 (2008), 1566-1603.
	\bibitem{Wan2} W. Wang,  On a conjecture of Anosov, {\it Adv. Math.} 230 (2012), 1597-1617.
	\bibitem{Wan12} W. Wang, Existence of closed geodesics on Finsler $n$-spheres.
	{\it  Nonlinear Anal.} 75 (2012), 751-757.
	
	\bibitem{Wan13} W. Wang, On the average indices of closed geodesics on positively curved Finsler spheres.
	{\it Math. Ann.} 355 (2013), 1049-1065.
	
	%\bibitem{West2005} C. Westerland,  Dyer-Lashof operations in the string topology of spheres and projective spaces, {\it Math. Z.} 250(3) (2005), 711-727.
	%\bibitem{West2007} C. Westerland,  String Homology of Spheres and Projective Spaces, {\it Algebr. Geom. Topol.} 7 (2007), 309-325.
	\bibitem{XL15} Y. Xiao and Y. Long, Topological structure of non-contractible loop space
	and closed geodesics on real projective spaces with odd dimensions. {\it Adv. Math.} 279 (2015), 159-200.
	\bibitem{VS1976} M. Vigu\'{e}-Poirrier and D. Sullivan, The homology theory of the closed geodesic problem,
	{\it J. Diff. Geom.} 11 (1976), 663-644.
	%\bibitem{Zil1977} W. Ziller,  The free loop space of globally symmetric spaces, {\it Invent. Math.} 41 (1977), 1-22.
	\bibitem{Zil83} W. Ziller,  Geometry of the Katok examples, {\it Ergod. Th. Dyn. Sys.} 3 (1983), 135-157.
\end{thebibliography}
\end{document}